\documentclass[10pt,reqno]{amsart}
\usepackage{amssymb}
\usepackage[all]{xy}
\newtheorem{thm}{Theorem}[section]

\newtheorem{prop}[thm]{Proposition}
\newtheorem{cor}[thm]{Corollary}

\theoremstyle{definition}
\newtheorem{dfn}[thm]{Definition}

\newtheorem{ex}[thm]{Example}

\theoremstyle{remark}
\newtheorem{remark}[thm]{Remark}

\newtheorem{notation}[thm]{Notation}

\begin{document}


\title[Saturated actions by finite dimensional Hopf
$*$-algebras] {Saturated actions by finite dimensional Hopf
$*$-algebras on $C^*$-algebras}

\author[Ja A Jeong]{Ja A Jeong$^{\dagger}$}
\thanks{Research supported by KRF-ABRL-R14(2003-2008)$^{\dagger}$
 and  Hanshin University
Research Grant$^{\ddagger}$.}
\address{
Department of Mathematical Sciences and Research Institute of Mathematics\\
Seoul National University\\
Seoul, 151--747\\
Korea} \email{jajeong\-@\-snu.\-ac.\-kr }
\author[Gi Hyun Park]{Gi Hyun Park$^{\ddagger}$}

\address{
Department of Mathematics\\
Hanshin University\\
Osan, 447--791\\
Korea}
\email{ghpark\-@\-hanshin.\-ac.\-kr }


\begin{abstract} If a finite group action $\alpha$
on a unital $C^*$-algebra $M$ is saturated, the canonical
conditional expectation $E:M\to M^\alpha$ onto the fixed point
algebra is known to be of index finite type  with $Index(E)=|G|$
in the sense of Watatani. More generally if a finite dimensional
Hopf $*$-algebra $A$ acts on $M$ and the action is saturated, the
same is true with $Index (E)=\dim(A)$. In this paper we prove that
the converse is true. Especially in case $M$ is a commutative
$C^*$-algebra $C(X)$ and $\alpha$ is a finite group action, we
give an equivalent condition in order that  the expectation
$E:C(X)\to C(X)^\alpha$ is of index finite type, from which we
obtain that $\alpha$ is saturated if and only if $G$ acts freely
on $X$.
 Actions by compact groups are also considered to
  show that the gauge action $\gamma$ on a graph
  $C^*$-algebra $C^*(E)$ associated with a locally finite directed
  graph $E$  is   saturated.
\end{abstract}

\maketitle

\setcounter{equation}{0}

\section{Introduction}

  It is known   \cite{Ro} that if
 $\alpha$ is an action by a compact group $G$ on a
$C^*$-algebra $M$, the fixed point algebra $M^\alpha$ is
isomorphic to a hereditary subalgebra $e(M\times_\alpha G)e$ of
the crossed product $M\times_\alpha G$ for a projection $e$  in
the multiplier algebra of $M\times_\alpha G$.
 If $e(M\times_\alpha G)e$ is  full in
 $M\times_\alpha G$ (that is, $e(M\times_\alpha G)e$
 generates $M\times_\alpha G$ as a closed two-sided ideal),
 the action $\alpha$ is said to be saturated
 (the notion of saturated action was introduced
 by Rieffel \cite[Chap.7]{Ph1}). Every action $\alpha$  with a
 simple crossed product $M\times_\alpha G$ is obviously saturated.

  On the other hand, an action of a
  finite dimensional Hopf $*$-algebra $A$ on a unital $C^*$-algebra
  $M$ is considered in \cite{SP}  and it is shown that if the action is
  saturated, the canonical
conditional expectation $E :M\to M^A$ onto the fixed point algebra
$M^A$ is of index finite type in the sense of Watatani \cite{Wa}
and   $Index(E)=(\dim A) 1$.
 The main purpose of the present  paper is to prove that the
 converse is also true.
 We see from our result that for an action $\alpha$ by a finite group $G$,
 $\alpha$ is saturated if and only if the canonical expectation
 $E:M\to M^\alpha$ is of index-finite type with index
 $Index(E)=|G|$.

 Besides, we consider actions by compact groups
 to study the saturation property of a gauge action $\gamma$
 on a $C^*$-algebra $C^*(E)$ associated with a locally finite directed
 graph $E$ with no sinks or sources.
 This paper is organized as follows.

 In section 2, we  review  the $C^*$-basic construction  from
 \cite{Wa} and finite dimensional Hopf $*$-algebras
 from \cite{SP} setting up our notations. Then
   we prove in section 3
that if $A$ is a finite dimensional Hopf $*$-algebra acting on a
unital $C^*$-algebra $M$ such that $E:M\to M^A$ is of index finite
type with $Index(E)=(dim
 A)1$, then the action is saturated (Theorem~\ref{thm-hopf}).

In section 4, we deal with the crossed product $M\times_\alpha G$
by a finite group in detail and give other equivalent conditions
in order that $\alpha$ be saturated. From the conditions one
easily see that an action with the Rokhlin property \cite{Iz} is
always saturated.
 Also we shall show that if $M$ has the cancellation,  an action with
 the tracial Rokhlin property \cite{OP} on $M$ is  saturated.

 Note that even for an action $\alpha$ by the finite group $\mathbb Z_2$,
 the expectation $E:M\to M^\alpha$ may not be of index
 finite type in general \cite[Example 2.8.4]{Wa}.
   For a commutative  $C^*$-algebra $C(X)$ and a finite group action
   $\alpha$, we give a necessary and sufficient
   condition that $E:C(X)\to C(X)^\alpha$ is of index finite
 type  (Theorem~\ref{thm-comm}) and provide a formula for $Index(E)$.
 Then as a corollary we obtain that $\alpha$ is saturated if and only if
  $G$ acts freely on $X$.

In section 5, we consider a compact group action $\alpha$  and
investigate the ideal $\mathcal J_\alpha$ of $M\times_\alpha G$
generated by the hereditary subalgebra $e(M\times_\alpha G)e$.
Then we apply the result on $\mathcal J_\alpha$ to the gauge
action on a graph $C^*$-algebra in section 6. As a generalization
of the Cuntz-Krieger algebras \cite{CK}, the  class of graph
$C^*$-algebras $C^*(E)$ associated with directed graphs $E$ has
been studied in various directions by considerably many authors
(for example see the bibliography in the book \cite{Ra} by
Raeburn).
 In \cite{KP}, Kumjian and Pask
show among others that if $\gamma$ is the gauge action on
$C^*(E)$,  then $C^*(E)^\gamma$ is stably isomorphic to the
crossed product $C^*(E)\times_\gamma \mathbb T$, which was done by
hiring the notions of skew product of graphs and groupoid
$C^*$-algebras.   In Theorem~\ref{thm-main-2} we shall directly
show that the gauge action is actually saturated (this implies
that $C^*(E)^\gamma$ and $C^*(E)\times_\gamma \mathbb T$ are
stably isomorphic).

\vskip 1pc

\section{Preliminaries}

\vskip .5pc   \noindent {\bf Watatani's index theory for
$C^*$-algebras}. In \cite{Wa}, Watatani developed the index theory
for $C^*$-algebras, and here we briefly review the basic
construction $C^*(B,e_A)$. Let $B$ be a $C^*$-algebra and $A$ its
$C^*$-subalgebra containing the unit of $B$. Let $E:B\to A$ be a
faithful conditional expectation. If there exist finitely many
elements $\{v_i\}_{i=1}^n$ in $B$ satisfying the following
$$b=\sum_i E(bv_i)v_i^*=\sum_i v_iE(v_i^*b), \ \text{ for every }
b\in B,$$ $E$ is said to be of \emph{index-finite type} and
$\{(v_i,v_i^*)\}_{i=1}^n$ is called a \emph{quasi-basis} for $E$.
The positive element $\sum_i v_iv_i^*$ is then the \emph{index} of
$E$, $Index(E)$, which is known to be an element in the center of
$B$ and does not depend on the choice of quasi-bases for $E$
(\cite[Proposition 1.2.8]{Wa}).
 Let $\mathcal B$ be the completion of the pre-Hilbert module
 $\mathcal B_0=\{\eta(b)\mid b\in B\}$ over $A$ with an
 $A$-valued inner product
 $$\langle \eta(x),\eta(y)\rangle = E(x^*y),\ \eta(x),\eta(y)\in
 \mathcal B_0.$$
 Let $\mathcal L_A(\mathcal B)$  be the $C^*$-algebra of all (right)
 $A$-module homomorphisms on $\mathcal B$ with
 adjoints. For $T\in \mathcal L_A(\mathcal B)$, the norm
 $\|T\|=\sup\{\|Tx\|:  \|x\|=1\}$ is always bounded.
 Each $b\in  B$ is regarded as an  operator $L_b$ in
  $\mathcal L_A(\mathcal B)$
 defined by $L_b(\eta(x))=\eta(bx)$ for $\eta(x)\in \mathcal B_0$.
 By  $e_A:\mathcal  B\to  \mathcal B$  we denote the  projection
 in $\mathcal   L_A(\mathcal B)$ such that
 $e_A(\eta(x))=\eta(E(x))$,
 $\eta(x)\in  \mathcal B_0$. Then the \emph{$C^*$-basic construction}
 $C^*(B,e_A)$ is the $C^*$-subalgebra of $\mathcal L_A(\mathcal B)$
 in which the linear span of elements $L_b e_A L_{b'}$ ($b,b'\in B$) is
 dense.

\vskip 1pc

\noindent {\bf Finite dimensional Hopf $*$-algebras.} As in
\cite{SP}, a finite dimensional  Hopf $*$-algebra is a finite
matrix pseudogroup of \cite{Wo}. We review from \cite{SP} the
definition and some  basic properties of a finite dimensional Hopf
$*$-algebra which we need in the following section.

\begin{dfn}{\rm (\cite[Proposition 2.1]{SP})} A finite
dimensional $C^*$-algebra is called a \emph{finite dimensional
 Hopf $*$-algebra} if there exist three linear maps,
$$\Delta: A\to A\otimes A,\ \ \epsilon: A\to \mathbb C,\ \text{ and }S:
A\to A$$ which satisfy the following properties
\begin{enumerate}
\item[(i)] $\Delta$(comultiplication) and $\epsilon$(counit) are
$*$-homomorphisms, and $S$(antipode) is a $*$-preserving
antimultiplicative involution,
 \item[(ii)] $\Delta(1)=1\otimes 1$, $\epsilon(1)=1$, $S(1)=1$,
 \item[(iii)] $(\Delta\otimes id)\Delta=(id\otimes \Delta)\Delta$,
 \item[(iv)] $(\epsilon\otimes id)\Delta=\Delta(\epsilon\otimes id)$,
 \item[(v)] $m(S\otimes id)(\Delta(a))=\epsilon(a)1=m(id\otimes S)(\Delta(a))$ for
  $a\in  A$, where $m:A\otimes A\to A$ is the multiplication.\end{enumerate}
\end{dfn}

\vskip 1pc

\begin{prop}{\rm (\cite{SP}, \cite{Wo})}\label{prop-tau-e}
Let $A$ be a finite dimensional Hopf $*$-algebra. Then  the
following properties hold.
 \begin{enumerate}

\item[(i)] For $a\in A$, with the notation   $\Delta(a) =\sum_i
a_i^L\otimes a_i^R,$   we have
\begin{align*}
 \sum_i \epsilon(a_i^L)a_i^R & = a = \sum_i \epsilon(a_i^R)a_i^L,\\
 \sum_i a_i^L S(a_i^R) & = \,\epsilon(a) 1 = \sum_i
S(a_i^L)a_i^R,\\
 \sum_i a_i^R S(a_i^L) &= \,\epsilon(a) 1 = \sum_i
S(a_i^R)a_i^L.
\end{align*}

\item[(ii)] There is a unique normalized trace (called the Haar
trace) $\tau$ on $A$ such that
$$ \sum_i\tau(a_i^L)a_i^R =\tau(a)1= \sum_i\tau(a_i^R)a_i^L ,\ a\in
A.$$

\item[(iii)] There exists a minimal central projection $e\in A$
(called the distinguished projection) such that $ae=\epsilon(a)e$,
$a\in A$. We have $$\epsilon(a)=1,\ S(e)=e, \text{ and }\
\tau(e)=(\dim A)^{-1}.$$

\end{enumerate}\end{prop}

\vskip 1pc

\section{Actions by finite dimensional Hopf $*$-algebras}

\noindent Throughout this section $A$ will be a finite dimensional
Hopf $*$-algebra.
 An \emph{action}  of $A$ on a unital $C^*$-algebra $M$ is
 a bilinear map $\cdot :A\times M\to M$ such that  for $a,b\in A$, $x,y\in M$,
\begin{align*} 1\cdot x& =x,\\
a\cdot 1 &=\epsilon(a)1,\\
ab\cdot x &=a\cdot(b\cdot x), \\
a\cdot xy &= \sum_i (a_i^L\cdot x)(a_i^R\cdot y),\\
(a\cdot x)^* &=S(a^*)\cdot x^*.\end{align*}
  Then the \emph{crossed product}
$M\rtimes A$ is the algebraic
 tensor product $M\otimes A$ as a vector space with the following
 multiplication and $*$-operation:
 \begin{align*} (x\otimes a)(y\otimes b)&:= \sum_j x(a_i^L\cdot
 y)\otimes a_i^Rb,\\
 (x\otimes a)^* &:= \sum_i (a_i^L)^*\cdot x^*\otimes
 (a_i^R)^*.
\end{align*}
Identifying  $a\in A$ with  $1\otimes a$ and   $x\in M$ with
$x\otimes 1$, we see \cite{SP} that $$M\rtimes A=span\{xa\mid x\in
M, \, a\in A\}.$$

\vskip 1pc For the definition of saturated action of $A$ on $M$,
refer to section 4 of \cite{SP}.

\vskip 1pc
\begin{prop}{\rm (\cite{SP})}\label{prop-hopf}
\label{hopf} Let  $M^A=\{x\in M\mid a\cdot x=\epsilon(a)x,\
\text{for all }a\in A\}$ be the  fixed point algebra  for the
action of $A$ on a unital $C^*$-algebra  $M$.
\begin{enumerate}
\item[(i)] The action
 is  saturated  if and only if $$M\rtimes
A=span\{xey\mid x,y\in M\},$$ where $e\in A$ is the distinguished
projection.

\item[(ii)] The map $E:M\to M^A$, $E(x)=e\cdot x$,  is a faithful
conditional expectation onto the fixed point algebra such that
$$E((a\cdot x)y)=E(x(S(a)\cdot y)), \ \ a\in A, \ x,y\in M.$$

\item[(iii)] The linear map $F: M\rtimes A \to M$,
$F(xa)=\tau(a)x$, is a faithful conditional expectation onto $M$.
\end{enumerate}
\end{prop}

\vskip 1pc \noindent Recall that $\mathcal M_0:=M$ is an
$M^A$-valued inner product module by $$\langle
\eta(x),\eta(y)\rangle_{M^A}=E(x^*y)$$ (here we use the
 convention  in \cite{Wa} for the inner product as in section 2).
 Since every norm bounded $M^A$-module map on $\mathcal M_0$
 extends uniquely to the Hilbert $M^A$-module $\mathcal M$,
 we may identify the $*$-algebra $End(\mathcal M_0)$ (in \cite{SP}) of norm
bounded right $M^A$-module endomorphisms of $\mathcal M_0$ having
an adjoint  with the $C^*$-algebra $\mathcal L_{M^A}(\mathcal M)$
explained in section 2.

\vskip 1pc \begin{remark}\label{compact}(\cite[Proposition
1.3.3]{Wa})  If $E:M\to M^A$ is of index-finite type, then
$$C^*(M,e_{M^A})= span\{L_x e_{M^A} L_y\mid
 x,y\in M\}=\mathcal L_{M^A}(\mathcal M).$$  In fact, we
 see  from the proof of  \cite[Proposition 2.1.5]{Wa}
 that $C^*(M,e_{M^A})$ contains the unit of
 $\mathcal L_{M^A}(\mathcal  M)$. Thus the ideal $span\{L_x e_{M^A} L_y\mid
 x,y\in M\}$ which is dense in $C^*(M,e_{M^A})$ must contain the unit of
 $\mathcal L_{M^A}(\mathcal M)$.

\end{remark}

\vskip 1pc
\begin{thm}\label{thm-hopf} Let $A$ be a finite dimensional Hopf $*$-algebra
acting on a unital $C^*$-algebra $M$. Then the following are
equivalent:
 \begin{enumerate} \item[(i)] The action is saturated.

 \item[(ii)] $E:M\to M^A$ is of index finite type with $Index(E)=(dim
 A)1$.
\end{enumerate}
\end{thm}

\begin{proof} (i)$\Longrightarrow$ (ii) is shown in \cite[Proposition 4.5]{SP}.

\noindent (ii)$\Longrightarrow$ (i). By Remark~\ref{compact},
$C^*(M,e_{M^A})=span\{L_x e_{M^A} L_y\mid x,y\in M\}$.
 Consider a map $\varphi: C^*(M,e_{M^A})\to M\rtimes A$ given by
 $$\varphi(\sum_i L_{x_i} e_{M^A} L_{y_i})=\sum_i  x_i  e y_i.$$
 To see that $\varphi$ is well defined, let $\sum_i L_{x_i} e_{M^A}
 L_{y_i}=0$. Then for each $z\in M$,
  $$(\sum_i L_{x_i} e_{M^A} L_{y_i})(\eta(z))=
  \sum_i \eta(x_i E(y_i z))=\eta(\sum_i
 x_i(e\cdot(y_iz)))=0,$$ hence by the injectivity of $\eta$ (\cite[2.1]{Wa}),
$\sum_i  x_i(e\cdot(y_iz)) =0$ in $M$.
 Since $(a\cdot x)e=axe$ for $a\in A$, $a\in M$ (see (7) of \cite{SP}),
  we thus have
  $$\sum_i  x_i(e\cdot(y_iz))e=\sum_i(x_ie y_i)ze=0$$ in $M\rtimes
  A$ for every $z\in M$,
  which then implies that
  $$(\sum_i  x_i e y_i )(zez')=0,  \ z,z'\in M.$$
 Particulary, $(\sum_i  x_i e y_i )(\sum_i  x_i e y_i )^*=0$,
  so that $\sum_i  x_i e y_i=0$ (in $M\rtimes A$). Thus $\varphi$ is well
  defined.
  It is tedious to show that $\varphi$ is a $*$-homomorphism such
  that the range   $\varphi(C^*(M,e_{M^A}))=MeM$  is an ideal
   of $M\rtimes A$;
  if $x,y,$ and $z\in M$ and $a\in A$, then
 $$(za)(xey) =(z(a\cdot x))ey\in  MeM.$$
  Hence it suffices to show that $\varphi (1)=1$.
 If $\{(u_i,u_i^*)\}_{i=1}^n$  is a quasi-basis for $E$, then
 $$\sum_i  L_{u_i}  e_{M^A}  L_{u_i^*} (\eta(z))=\sum_i \eta(u_i E(u_i^*
 z))  =\eta(z), \ \ z\in M,$$ which means that
 $\sum_i  L_{u_i}  e_{M^A}  L_{u_i^*}=1 \in
 \mathcal L_{M^A}(\mathcal M)$.
 Therefore by Proposition~\ref{prop-tau-e}(iii) and
 Proposition~\ref{prop-hopf}(iii)
 $$
 F(\varphi(1))=F(\sum_{i=1}^n  u_i e u_i^*)=\sum_i \tau(e) u_i u_i^*
 =\frac{1}{\dim A}\sum_i u_i u_i^*=1.$$
 Since $\varphi$ is a $*$-homomorphism, $\varphi(1)$ is a
 projection in $M\rtimes A$
 such that $F(1-\varphi(1))=0$. But $F$ is faithful, and
 $\varphi(1)=1$  follows.
\end{proof}

\vskip 1pc
\section{Actions by finite groups}

Throughout this section $G$ will denote a finite group. As is well
known the group $C^*$-algebra $C^*(G)$ generated by the unitaries
$\{\lambda_g\mid g\in G\}$ is a finite dimensional Hopf
$*$-algebra  with $$\Delta(\lambda_g)= \lambda_g\otimes
\lambda_g,\ \ \epsilon(\lambda_g)= 1,\ \ S(\lambda_g)=
\lambda_{g^{-1}}  \ \text{ for }\lambda_g\in C^*(G).$$
 The Haar trace $\tau$ is  given by $\tau(\lambda_g)=\delta_{\iota
 g}$, where $\iota$ is the identity of $G$, and the distinguished
 projection   is  $e=\frac{1}{|G|}\sum_g \lambda_g $.

Let $\alpha$  be an action of  $G$ on a unital $C^*$-algebra $M$.
Then it is easy to see  that
  $$\lambda_g\cdot x:= \alpha_g(x) \ \ \text{for } g\in G,\   x\in M,$$
  defines an action of  $C^*(G)$ on $M$. Furthermore
  $M\rtimes C^*(G)$ is nothing but the usual crossed product $M\times_\alpha
 G =span\{x\lambda_g\mid x\in M,\ g\in G\}$, and
 the expectations  $E:M\to M^\alpha(=M^{C^*(G)})$, $F:M\times_\alpha G\to M$
 of Proposition~\ref{prop-hopf} are given by
\begin{eqnarray}{\label{eqn-exp-fa}}
E (x)=\frac{1}{|G|}\sum_g \alpha_g(x)\ \text{ and } \  F(\sum_g
x_g \lambda_g)=x_{\iota} \ (x, x_g\in M,\  g\in G).
\end{eqnarray}

\noindent
 Note that for each   $\sum_h
x_h \lambda_h\in M\times_\alpha G$ and  $g\in G$,
\begin{eqnarray}\label{eqn-norm} \|x_g\|
=\|F((\sum_h x_h \lambda_h) \lambda_{g^{-1}})\|\leq \|(\sum_h x_h
 \lambda_h) \lambda_{g^{-1}}\|= \|\sum_h x_h  \lambda_h\|.\end{eqnarray}

 If $\mathcal J_\alpha$  denotes  the closed
ideal of $M\times_\alpha G$ generated by the distinguished
projection $e$, then Proposition~\ref{prop-hopf}(i) says that
$\alpha$ is saturated if
 and only if $\mathcal J_\alpha=M\times_\alpha G$.
We will see in Proposition~\ref{prop-J} that
 \begin{eqnarray}\label{eqn-J-alpha}\mathcal
J_\alpha\, =\, span \{\sum_g x\alpha_g(y) \lambda_g\mid x,y\in M\}
\,  =\, span \{\sum_g x\alpha_g(x^*) \lambda_g\mid x \in
M\}.\end{eqnarray}

\noindent The $*$-homomorphism $\varphi:C^*(M, e_{M^\alpha})\to
M\times_\alpha G$ we discussed in the proof of
Theorem~\ref{thm-hopf} can be rewritten as follows.
\begin{eqnarray}\label{eqn-varphi}
\varphi(L_x e_{M^\alpha} L_y)=  \frac{1}{|G|} \sum_g
x\alpha_g(y)\lambda_g, \ x,y\in M
\end{eqnarray}
 because $\varphi(L_x e_{M^\alpha} L_y)=xey$ and
 $e= \frac{1}{|G|} \sum_g  \lambda_g$.
   If $\{(u_i, u_i^*)\}$ is a
quasi-basis for $E$, we see from   $\sum_i L_{u_i} e_{M^\alpha}
L_{u_i^*}=1$ and (\ref{eqn-varphi}) that
\begin{eqnarray}\label{eqn-quasibasis} \varphi(1)=  \sum_i  u_i
e  u_i^* = \frac{1}{|G|}\sum_g (\sum_i
u_i\alpha_g(u_i^*))\lambda_g\end{eqnarray} is a projection in
$M\times_\alpha G$. Recall that $\varphi(1)=1$ holds if $\alpha$
is saturated.

\vskip 1pc

\begin{thm}\label{thm-main-1} Let $M$ be a unital $C^*$-algebra
 and $\alpha$ be an action of a finite group $G$ on $M$.
Then the following are equivalent:

\begin{enumerate}
\item[(i)] $\alpha$ is saturated, that is, $\mathcal J_\alpha=
M\times_\alpha G$.

\item[(ii)]   $E: M\to M^\alpha$  is of index finite type with
$Index(E)=|G|$.

\item[(iii)]  $E : M\to M^\alpha$  is of index finite type with
$Index(E)=|G|$ and
  \begin{eqnarray}\label{eqn-qb-condition} \sum_i u_i\alpha_g(u_i^*)=0, \
 g\neq \iota\end{eqnarray} for a quasi-basis
   $\{(u_i,u_i^*)\}$ for $E$.

\item[(iv)] There exist $\{b^j_g\in M\mid g\in G,\ 1\leq j\leq
m\}$ for some $m\geq 1$  such that
\begin{enumerate}
\item  $\alpha_g (b^j_h)= b^j_{gh},$ for $j=1,\dots, m$ and
$g,h\in G$.

\item  $\sum_j  b^j_g (b^j_h)^* = \delta_{gh}.$
\end{enumerate}

\item[(v)] For every $\varepsilon > 0$,  there exist $\{b^j_g\in
M\mid g\in G,\ 1\leq j\leq m\}$ for some $m\geq 1$  such that
\begin{enumerate}
\item  $ \sum_j\| \alpha_g (b^j_h) - b^j_{gh}  \| <
 \varepsilon, $

\item  $ \| \sum_j  b^j_g (b^j_h)^* -\delta_{gh} \|< \varepsilon.$
\end{enumerate}

\end{enumerate}
\end{thm}

\begin{proof} (i) $\Longleftrightarrow$ (ii) follows from
Theorem~\ref{thm-hopf}.

\noindent (i) $\Longrightarrow$ (iii). If $\{(u_i, u_i^*)\}$ is a
quasi-basis for $E$, we have from (\ref{eqn-quasibasis}) that
$\sum_i u_i\alpha_g(u_i^*)=0$ for $g\neq \iota$ since
$\varphi(1)=1$.

 (iii) $\Longrightarrow$ (ii). Obvious.

 (i) $\Longrightarrow$ (iv). Suppose
$\mathcal J_\alpha=M\times_\alpha G$. By (\ref{eqn-J-alpha}) there
exist $m\in \mathbb N$ and $b_j\in M$, $1\leq j\leq m$, such that
$$ \sum_g \big(\sum_j
b_j\alpha_g(b_j^*)\big)\lambda_g = 1.$$
 Thus
 \begin{eqnarray}\label{eqn-b}  \sum_j b_j b_j^* =1
 \ \text{ and }  \sum_j b_j\alpha_g(b_j^*) =0\ \text{ for }g\neq
 \iota.\end{eqnarray}
   Set $b^j_g:=\alpha_g(b_j).$ Then
\begin{align*}    \alpha_g (b^j_h)&  = \alpha_{g}
(\alpha_{h}(b_j))=\alpha_{gh}(b_j)= b^j_{gh},\\
    \sum_j b^j_g (b^j_h)^*
  & = \sum_j \alpha_g(b_j)\alpha_h(b_j^*) \\
  & = \alpha_g\big(\sum_j b_j \alpha_{g^{-1}h}(b_j^*)\big)  \\
  & =\delta_{gh}\  \text{by } (\ref{eqn-b}).
\end{align*}

 (iv) $\Longrightarrow$ (v). Obvious.

 (v) $\Longrightarrow$ (i). Let $\varepsilon>0$ and let
  $\{ b^j_g\in M\mid g\in G,\ 1\leq j\leq m\}$ satisfy
   (a) and (b) of (v).
  Note that (b) implies   $\|b_g^j\|< 1+\varepsilon$ for $g\in G,\
1\leq j\leq  m$. Indeed from $\big\|\, \|\sum_j b^j_g (b^j_g)^*\|
-1\big\|\leq \|\sum_j b^j_g (b^j_g)^* -1\|<\varepsilon,$ we have
 $\|b_g^j\|^2\leq \|\sum_j b_g^j(b_g^j)^*\|<1+\varepsilon$.
 Then
   \begin{align*}
   &\ \|\sum_{h,j} (\sum_g b^j_h\alpha_g((b^j_h)^*) \lambda_g)-|G|\,\| \\
   = &\ \|\sum_g(\sum_{h,j} b^j_h\alpha_g((b^j_h)^*)\lambda_g-|G|\,\| \\
   \leq  &\ \|\sum_{h,j} b^j_h (b^j_h)^* -|G|\,\|+\|\sum_{g\neq \iota} (\sum_{h,j}
               b^j_h\alpha_g((b^j_h)^*)\lambda_g)\|\\
   \leq & \ \sum_h\|\sum_{j} b^j_h (b^j_h)^* -1\,\|+\sum_{g\neq \iota} \sum_{h}
            \| \sum_{j}b^j_h\alpha_g((b^j_h)^*)\|\\
    < &\ \varepsilon |G| +\sum_{g\neq \iota}\sum_{h}
    \|\sum_{j}b^j_h\big(\alpha_g((b^j_h)^*)-(b^j_{gh})^*\big)\|+
    \sum_{g\neq \iota}\sum_{h}\|\sum_{j}b^j_h(b^j_{gh})^*\|\\
  < & \ \varepsilon \Big(|G|+  |G|^2\max_{g,j}\|b^j_g\| + |G|^2 \Big)\\
  < & \ \varepsilon \Big(|G|+  |G|^2(1+\varepsilon)  + |G|^2 \Big).
     \end{align*}
 Since $\sum_{h,j} (\sum_g b^j_h\alpha_g((b^j_h)^*) \lambda_g)\in
 \mathcal J_\alpha$ and $\varepsilon$ can be chosen to be arbitrarily small,
 we conclude that $\mathcal J_\alpha=M\times_\alpha G$.
\end{proof}

\vskip 1pc

\begin{ex}\label{exam-nonsat}
Let $w= \left(%
\begin{array}{cc}
  z_1 & 0 \\
  0 & z_2 \\
\end{array}%
\right)$  be a unitary with $w^n=1$ and define an automorphism
$\alpha$ on $M_2(\mathbb C))$  by $\alpha(a)=waw^*$, $a\in
M_2(\mathbb C)$. We will show that $\alpha$ is saturated if and
only if $z_2=-z_1$.
  For this, recall from (\ref{eqn-J-alpha}) that  $\alpha$ is
  saturated if and only if   there exist $x_j\in
M_2(\mathbb C)$, $1\leq i\leq m$, satisfying
\begin{eqnarray}\label{ex-matrix-0} \sum_{k=0}^{n-1}\sum_{j=1}^m
x_j\alpha^k(x_j^*)\lambda_k =1_{M_2(\mathbb C)}.\end{eqnarray}
 Hence, particularly for $k=0,1$, we have
 $$  \sum_j  x_j x_j^*= \left(%
\begin{array}{cc}
  1 & 0 \\
  0 & 1 \\
\end{array}%
\right)\ \text{ and }\  \sum_j x_j\alpha (x_j^*) = \sum_j x_jw
x_j^* w^* =\left(%
\begin{array}{cc}
  0 & 0 \\
  0 & 0 \\
\end{array}%
\right).$$
With  $x_j= \left(%
\begin{array}{cc}
  a_j & b_j \\
  c_j & d_j \\
\end{array}%
\right)$  and $z_i= e^{i\theta_i}$, $i=1,2,$ this means
$$  \sum_j \left(%
\begin{array}{cc}
  |a_j|^2+|b_j|^2 & a_j \bar{c_j}+b_j\bar{d_j} \\
  c_j \bar{a_j}+d_j\bar{b_j} & |c_j|^2+|d_j|^2 \\
\end{array}%
\right)= \left(%
\begin{array}{cc}
  1 & 0 \\
  0 & 1 \\
\end{array}%
\right),$$
\begin{eqnarray}\label{eqn-b-notzero} \sum_j \left(%
\begin{array}{cc}
   |a_j|^2+e^{i(\theta_2-\theta_1)}|b_j|^2 &
   e^{i(\theta_1-\theta_2)}a_j \bar{c_j}+b_j\bar{d_j} \\
   |a_j|^2+e^{i(\theta_2-\theta_1)}d_j\bar{b_j} &
   e^{i(\theta_1-\theta_2)}a_j \bar{c_j}+|d_j|^2 \\
\end{array}%
\right)  =\left(%
\begin{array}{cc}
  0 & 0 \\
  0 & 0 \\
\end{array}%
\right).
\end{eqnarray}
 Therefore, by comparing (1,1) entries of each matrices,
 it follows that if $\alpha$ is
saturated, then   there exist positive real numbers $a\,(=\sum_j
|a_j|^2)> 0$, $b\,(=\sum_j |b_j|^2)> 0$ such that
\begin{eqnarray}\label{ex-matrix-2} a+b =1   \ \text{ and }
 a+e^{i(\theta_2-\theta_1)}
b  =0.\end{eqnarray}
 Note that $b\neq 0$ since $b=0$ implies $a=0$ from
 $\sum_j \big(|a_j|^2+e^{i(\theta_2-\theta_1)}d_j\bar{b_j}\big) =0$
  in (\ref{eqn-b-notzero}).
  There are three possible cases for $\theta_1$,
 $\theta_2$ as follows.
 \begin{enumerate} \item[(i)]
 If $\theta_2-\theta_1\equiv 0(\text{mod }
2\pi)$, that is, $\alpha$ is trivial, then  (\ref{ex-matrix-2}) is
not possible.

 \item[(ii)] If $\theta_2-\theta_1\equiv  \pi (\text{mod }
2\pi)$, then $$x_1=\frac{1}{\sqrt{2}}\left(%
\begin{array}{cc}
  1 & 1 \\
  0 & 0 \\
\end{array}%
\right) \ \text{ and }\ x_2=\frac{1}{\sqrt{2}}\left(%
\begin{array}{cc}
  0 & 0 \\
  1  & 1  \\
\end{array}%
\right)$$ satisfy (\ref{ex-matrix-0}) with $m=2$. Thus   $\alpha$
is saturated.

 \item[(iii)] If $\theta_2-\theta_1\neq 0,\,\pi (\text{mod }
2\pi)$, then (\ref{ex-matrix-2}) is not possible for any
   $a$, $b>0$. Hence
$\alpha$ is not saturated.

\end{enumerate}

\end{ex}

\vskip 1pc
\begin{remark}
 Let $\alpha$, $\beta\in Aut(M)$ satisfy
$\alpha^n=\beta^n=id_M$ for some $n\geq 1$. If there is a unitary
$u\in M$ such that  $\beta=Ad(u)\circ\, \alpha$, then $\alpha$ and
$\beta$ are said to be \emph{exterior equivalent}, and if this is
the case the crossed
 products are isomorphic, $M\times_\alpha G\cong M\times_\beta G$,
 \cite[p.45]{Ph1}.
   Example~\ref{exam-nonsat} says that
  the property of being saturated may not be preserved
  under exterior equivalence.
  Also the case (iii) of Example~\ref{exam-nonsat} above with
$w=diag(\lambda,\bar{\lambda})$, $\lambda=e^{\frac{2\pi}{3}}$
(hence $\theta_1-\theta_2=\frac{2\pi}{3}-\frac{4\pi}{3} \neq
\pi(\text{mod } 2\pi))$, shows that $Index(E)<|G|$ is possible
even when $E$ is of index-finite type. In fact,
 if
$u_1= \left(%
\begin{array}{cc}
  1 & 0 \\
  0 & 1 \\
\end{array}%
\right)$ and $u_2= \left(%
\begin{array}{cc}
  0 & 1 \\
  1 & 0 \\
\end{array}%
\right)$, then  $\{(u_i,u_i^*)\}_{i=1}^2$ forms a quasi-basis for
$E$, but $Index(E)=2<|\mathbb Z_3|$.

\end{remark}

\vskip 1pc
\begin{remark}
  Recall that the Rokhlin property and the tracial Rokhlin
  property (weaker than the Rokhlin property)
 are defined as follows and considered intensively in \cite{Iz} and \cite{OP},
 respectively:
 \begin{enumerate}
 \item[(a)](\cite{Iz}) $\alpha$ is said to have the \emph{Rokhlin property}  if
 for every finite set $F\subset M$, every $\varepsilon>0$,  there are mutually
orthogonal projections $\{e_g \mid g\in G\}$ in $M$  such that
\begin{enumerate}
\item[(i)] $\|\alpha_g(e_h)-e_{gh}\|<\varepsilon$ for $g,h\in G.$
\item[(ii)] $\|e_gx-xe_g\|<\varepsilon $ for $g\in G$ and all
$x\in F$. \item[(iii)] $\sum_{g\in G}  e_g=1$.
\end{enumerate}

\item[(b)](\cite{OP}) $\alpha$ is said to have the  \emph{tracial
Rokhlin property}  if for every finite set $F \subset M$, every
$\varepsilon > 0$, every $n \in \mathbb{N}$, and every nonzero
positive element $x \in M$, there are mutually orthogonal
projections $\{e_g \mid g\in G\}$ in $M$ such that:
\begin{enumerate}
 \item[(i)] $\| \alpha_g (e_h) - e_{gh} \| < \varepsilon$ for $g,h
\in G$.
 \item[(ii)] $\| e_g x - x e_g \| < \varepsilon$ for $g \in
G$ and all $x \in F$.
 \item[(iii)]  With $e = \sum_{g \in G} e_g$,
the projection $1 - e$ is Murray-von Neumann equivalent to a
projection in the hereditary subalgebra of $M$ generated by $x$.
\end{enumerate}

\end{enumerate}

\end{remark}

\vskip 1pc

\noindent The following proposition is actually   observed  in
\cite[Lemma 1.13]{OP}, and we put a proof for reader's
convenience.

\vskip 1pc

\begin{prop}
  Let $M$ be a unital $C^*$-algebra and $\alpha$ be an action
of a discrete group $G$ on $M$. Suppose that for every
$\varepsilon>0$ and every finite subset $F\subset M$, there exist
a family of projections $\{e_g\}_{g\in G}$ such that
\begin{enumerate} \item $\|\alpha_g(e_h)-e_{gh}\|<\varepsilon.$
\item $\|e_gx-xe_g\|<\varepsilon$  for each $x\in F$.
\end{enumerate}  Then  $\alpha$ is an outer action.
\end{prop}

\begin{proof}
 Suppose there is a unitary $u\in
M$ such that $\alpha_g(x)= uxu^*$ for every $x\in M$ ($g\neq
\iota$). Put $F=\{u\}$ and $0<\varepsilon< 1/4$. Then there exist
mutually orthogonal projections $\{e_g\}_{g\in G}$ such that
$\|\alpha_g(e_h)-e_{gh}\|<\varepsilon< 1/4$. Thus $\|e_g u-u
e_g\|<\varepsilon< 1/4$. Then $\|\alpha_g(e_\iota)-ue_\iota
u^*\|=0$. But
\begin{align*}\|\alpha_g(e_\iota)-ue_\iota u^*\|&=
\|\alpha_g(e_\iota)-e_g+e_g-e_\iota+e_\iota-ue_\iota u^*\|\\
& \geq \|e_g-e_\iota\|-\|\alpha_g(e_\iota)-e_g\|-\|e_\iota-ue_\iota 1u^*\|\\
& \geq 1-\frac{1}{4}-\frac{1}{4}=\frac{1}{2}, \end{align*}
 which is a contradiction.
\end{proof}

\vskip 1pc

\vskip 1pc
\begin{remark}\label{remark} If $M\times_\alpha G$ is simple,
$\alpha$ is obviously saturated, and this is the case  if $G$ is a
finite group, $M$ is $\alpha$-simple, and $\tilde{\mathbb T}
(\alpha_g)\neq \{1\}$ for all $g\neq \iota$ (\cite[Theorem
3.1]{Ki}). In particular, $\alpha$ is saturated if $M$ is simple
and $\alpha$ is outer.

  But for a nonsimple $M$, this may not hold.
  In fact, if $\alpha$ is  an outer action of $\mathbb Z_n$ on $M$ and $u$
  is a unitary in $M$ with $u^n=1$
 such that the action $Ad(u)$ on $M$ is not saturated (as in
 Example~\ref{exam-nonsat}), then the action $\alpha\oplus Ad(u)$
 on $M\oplus M$  is outer but not saturated.
 \end{remark}

 \vskip 1pc \noindent  Now we show that
 if $\alpha$ satisfies the Rokhlin property (or satisfies
 the tracial Rokhlin property and
$M$ has cancellation) then $\alpha$ is saturated. For this we
first review the cancellation property of $C^*$-algebras. For
projections $p$, $q$ in a $C^*$-algebra, we write $p\perp q$ if
$pq=0$, and $p\sim q$ if they are Murrey-von Neumann equivalent.

\vskip 1pc
\begin{dfn} A unital $C^*$-algebra
$M$ has the \emph{cancellation} if, whenever $p$, $q$, $r$ are
projections in $M_n(M)$ for some $n$, with $p\perp r$, $q\perp r$,
and $(p+r)\sim (q+r)$, then $p\sim q$.
\end{dfn}

\vskip 1pc
\begin{remark} \begin{enumerate}
\item If $M$ has the cancellation and $p$, $q$ are projections in
$M$ such that $(1-p)\sim (1-q)$, then $p\sim q$
(\cite[V.2.4.14]{Bl-2}).

\item  It is well known that every $C^*$-algebra with stable rank
one has the cancellation (\cite[V.3.1.24]{Bl-2}).
\end{enumerate}
\end{remark}

\vskip 1pc

\begin{prop} Let $\alpha$ be an action of a finite group $G$ on a
unital $C^*$-algebra $M$. Then $\alpha$ is saturated if one of the
following holds.
 \begin{enumerate}
 \item[(i)] $\alpha$ has the Rokhlin property.
 \item[(ii)] $\alpha$ has the tracial Rokhlin property and $M$
has the cancellation.
 \end{enumerate}
\end{prop}
\begin{proof} (i)  For an $\varepsilon>0$, there exist mutually orthogonal
projections $\{e_g\}_g$ such that $\sum_g e_g=1$ and
$\|\alpha_g(e_h)-e_{gh}\|<\varepsilon.$ Then, with $m=1$, the
elements  $b^1_g:=e_g$ satisfy (v) of Theorem~\ref{thm-main-1}.

 (ii) Now suppose $\alpha$ has the tracial Rokhlin property and
 $M$ has the cancellation.
 We shall show that $\mathcal J_\alpha$ contains
 the unit of $M\times_\alpha G$.
 Let $0<\varepsilon<1$. For each $g\in G$, choose mutually orthogonal
 projections $\{e_h^g\}_{h\in G}$  such that
$$\|\alpha_k(e^g_h)-e^g_{kh}\|<\frac{\varepsilon}{2|G|^2},$$ and
put $e^g:=\sum_{h\in G} e^g_h$.  If $e^g=1$, for some $g$, then
$b^1_h:=e^g_h$ ($h\in G$) will satisfy (v) of
Theorem~\ref{thm-main-1} as in (i). If $e^g\neq 1$ for every $g\in
G$, then by the tracial  Rokhlin property of $\alpha$ there exist
mutually orthogonal projections $\{f^g_h\}_{h\in G}$ in $M$ such
that
$$\|\alpha_k(f^g_h)-f^g_{kh}\|<\frac{\varepsilon}{2|G|^2}$$  and
$$(1-\sum_{h\in G} f^g_h)\sim (e^g)'< e^g$$ for a subprojection
$(e^g)'$ of $e^g$. Put $f^g:=\sum_h f^g_h$. Then since $M$ has
cancellation, it follows that $f^g\sim
\big(1-(e^g)'\big)>(1-e^g)$. Let $v_g\in M$ be a partial isometry
satisfying
$$v_g^*v_g=f^g,\ \ v_gv_g^*=1-(e^g)', $$ and set
$$x_g:=\frac{1}{|G|} \sum_{k,h} \Big((e^k_h\alpha_g(e^k_h)
+(1-e^k)v_k f^k_h\alpha_g(f^k_h \alpha_{g^{-1}}(v_k^*)) \Big), \
g\in G.$$
 Now we show that the element $x:=\sum_g x_g\lambda_g\in \mathcal
 J_\alpha$ satisfies $\|x-1\|<\varepsilon$.
In fact,  for $g\neq \iota$,
\begin{align*}\|x_g\|&\leq \frac{1}{|G|} \sum_{k,h} \|(e^k_h\alpha_g(e^k_h)
+(1-e^k)v_k f^k_h\alpha_g(f^k_h \alpha_{g^{-1}}(v_k^*))\|\\
&\leq \frac{1}{|G|}
\sum_{k,h}\big(\|(e^k_h\alpha_g(e^k_h)\|+\|f^k_h\alpha_g(f^k_h)\|\big)\\
&\leq \frac{1}{|G|}
\sum_{k,h}\big(\|e^k_h(\alpha_g(e^k_h)-e^k_{gh})\|+\|e^k_h
e^k_{gh}\|+\|f^k_h(\alpha_g(f^k_h)-f^k_{gh})\|+\|f^k_h
f^k_{gh}\|\big)\\
&\leq \frac{1}{|G|} \sum_{k,h}(\frac{\varepsilon}{2|G|^2}
+\frac{\varepsilon}{2|G|^2})\\
&=\frac{\varepsilon}{|G|}\end{align*}
 and
\begin{align*}x_\iota&=\frac{1}{|G|}(\sum_k
e^k+\sum_k(1-e^k)v_kf^kv_k^*)\\
&=\frac{1}{|G|}\big(\sum_k
 e^k+\sum_k(1-e^k)(1-(e^k)')\big)\\
 &=1.\end{align*}
\end{proof}

\vskip 1pc
 For the rest of this section we consider a finite group action on a commutative
 $C^*$-algebra  $C(X)$.
 If $G$ acts on a compact Hausdorff space $X$,
 it induces an action, say $\alpha$, on $C(X)$ by
 $$\alpha_g(f)(x)=f(g^{-1}x), \ f\in C(X).$$
For each $x\in X$, let $G_x=\{ g\in G : gx=x \}$ be the isotropy
group of $x$ and for a subgroup $H$ of $G$ ($H<G$), put
$$
X_H = \{ x \in X : G_x=H\}.
$$
It is readily seen that $X_H$ and $X_{H'}$ are disjoint if $H \neq
H'$, and $X$ is partitioned as
$$
X = \bigcup_{H<G} X_H.
$$

\vskip 1pc
\begin{thm}\label{thm-comm}
Let $X$ be a compact Hausdorff space and $G$ a finite group acting
on $X$.   If $\alpha$ is the induced action of $G$ on $C(X)$, the
following are equivalent:
\begin{enumerate}
    \item[(i)] $E :C(X)\to C(X)^\alpha$ is of index finite type.
    \item[(ii)] $X_H$ is closed for each $H<G$.
\end{enumerate}
Moreover, if this is case the index of $E $ is $Index(E
)=\sum_{H<G} \frac{|G|}{|H|} \chi_{X_H},$ where $\chi_{X_H}$ is
the characteristic function on $X_H$.
\end{thm}
\begin{proof}
(i) $\Longrightarrow$ (ii). If $E $ is of index-finite type and
$\{(u_i, u_i^*)\}_{i=1}^k$ is a quasi-basis for $E$, then
$$
\sum_i u_i E (u_i^* f)=f,
$$
that is,
\begin{equation} \label{eqn-f}
\frac{1}{|G|}\sum_i u_i(x) \Bigl(\sum_{g\in G} u_i^*(g^{-1}x)
f(g^{-1}x)\Bigr)=f(x),
\end{equation}
 for $f \in C(X)$ and $x\in X $.
For each $x\in X$, choose a continuous function $f_x \in C(X)$
satisfying $f_x|_{Gx\setminus \{x\}}\equiv 0$ and $f_x(x)=1$.
 Then (\ref{eqn-f}) with $f_x$ in place of $f$ gives
\begin{equation} \label{eqn-f-x}
\frac{1}{|G|}\sum_i u_i(x) \Bigl(\sum_{g\in G} u_i^*(g^{-1}x)
f_x(g^{-1}x)\Bigr)=f_x(x),
\end{equation}
and so we have
\begin{equation} \label{eqn-indexvalue}
\frac{|G_x|}{|G|}\sum_i u_i(x)  u_i^*(x) =1.
\end{equation}
  To show that each $X_H$ is closed, let $\{x_n \in X_H :n=1,2,
\ldots \}$ be a sequence of elements in $X_H$ with limit $x\in
X_{H'}$. Then (\ref{eqn-f}) gives
\begin{align*}
f_x(x_n) &= \frac{1}{|G|}\sum_i u_i(x_n) \Bigl(\sum_{g\in G}
u_i^*(g^{-1}x_n) f_x(g^{-1}x_n)\Bigr) \\
&= \frac{1}{|G|}\sum_i u_i(x_n) \Bigl(\sum_{g\in H}
u_i^*(g^{-1}x_n) f_x(g^{-1}x_n) + \sum_{g\not\in H}
u_i^*(g^{-1}x_n) f_x(g^{-1}x_n)\Bigr) \\
&= \frac{1}{|G|}\sum_i u_i(x_n) \Bigl(|H| u_i^*(x_n) f_x(x_n) +
\sum_{g\not\in H} u_i^*(g^{-1}x_n) f_x(g^{-1}x_n)\Bigr).
\end{align*}
Taking the limit as $n\to \infty$, we have
\begin{align*}
f_x(x)&= \frac{1}{|G|}\sum_i u_i(x) \Bigl(|H| u_i^*(x) f_x(x) +
\sum_{g\not\in H}
u_i^*(g^{-1}x) f_x(g^{-1}x)\Bigr) \\
&= \frac{1}{|G|}\sum_i u_i(x) \Bigl(|H| u_i^*(x) f_x(x) +
|H'\setminus H| u_i^*(x) f_x(x)\Bigr)\\
&= \frac{|H|+ |H'\setminus H|}{|G|}\sum_i u_i(x) u_i^*(x) f_x(x).
\end{align*}
Therefore, comparing with (\ref{eqn-indexvalue}), we obtain
$$
|H'|=|H|+ |H'\setminus H|
$$ since
$G_x=H'$ and $f_x(x)=1$. Hence  $$ H\subset H'.$$
 On the other hand, since $G_{x_n}=H$, again by
 (\ref{eqn-indexvalue}),
$ \frac{|H|}{|G|}\sum_i u_i(x_n)  u_i^*(x_n) =1 $ with the limit $
\frac{|H|}{|G|}\sum_i u_i(x)  u_i^*(x) =1$ as $n\to \infty$.
 But also $\frac{|H'|}{|G|}\sum_i u_i(x)  u_i^*(x) =1$ by (\ref{eqn-indexvalue}),
 and thus $|H|=|H'|$.
Consequently we have $$H=H'$$ because  $H\subset H'.$   This shows
that $X_H$ is closed.

(ii) $\Longrightarrow$ (i). Assume that $X_H$ is closed for every
subgroup  $H$ of $G$. Then $X_H$ is open since there are only
finitely many such subsets. Let $\mathcal{U}_H=\{U_{H,i_H} :
i_H=1,2,\ldots,n_H\}$ be an open covering of $X_H$ such that
$$x\in U_{H,i_H}\ \Longrightarrow\ g^{-1}x \not\in U_{H,i_H}
\text{ or }  g^{-1}x \not \in X_H \text{ whenever } g^{-1}x \neq
x.$$ Let $\{ v_{H,i_H} \}$ be a partition of unity subordinate to
$\mathcal{U}_H$. We understand that the domain of $v_{H,i_H}$ is
$X$ by assigning $0$ to $x \not\in X_H$. Let
$u_{H,i_H}=\sqrt{v_{H,i_H}}$.

We claim that
\begin{eqnarray}\label{eqn-comm-qb}
\bigg\{\big(\sqrt{\frac{|G|}{|H|}} \,u_{H,i_H},\,
\sqrt{\frac{|G|}{|H|}}\,u_{H,i_H}^*\big) : H<G, i_H=1,2,\ldots,n_H
\bigg\}
\end{eqnarray}
is a quasi-basis for $E$. For $f\in C(X)$ and $x\in X$, let $F<G$
and $1\leq j \leq n_F$ be such that $x\in X_F$ and $x \in
U_{F,j}$. Then
\begin{align*}
&\  \sum_{H<G}\sum_{i_H=1}^{n_H} \biggl(
 \sqrt{\frac{|G|}{|H|}}\,u_{H,i_H} E\big(\sqrt{\frac{|G|}{|H|}}
 \,u_{H,i_H}^* f \big) \biggr)(x) \\
=&\ \frac{1}{|G|}\sum_{H<G}\sum_{i_H=1}^{n_H} \biggl(
\sqrt{\frac{|G|}{|H|}}u_{H,i_H}(x) \Bigl(\sum_{g\in G}
\sqrt{\frac{|G|}{|H|}}u_{H,i_H}^*(g^{-1}x) f(g^{-1}x)\Bigr)\biggr) \\
=&\ \sum_{{i_F}=1}^{n_F}\frac{1}{|F|} u_{F,{i_F}}(x)
\Bigl(\sum_{g\in G}
u_{F,{i_F}}^*(g^{-1}x) f(g^{-1}x)\Bigr) \\
=&\ \frac{1}{|F|} \sum_{{i_F}=1}^{n_F} u_{F,{i_F}}(x)
\Bigl(\sum_{g\in F}
u_{F,{i_F}}^*(g^{-1}x) f(g^{-1}x)\Bigr) \\
=&\ \frac{1}{|F|} \sum_{{i_F}=1}^{n_F} u_{F,{i_F}}(x) | F |
u_{F,{i_F}}^*(x) f(x) \\
=&\  \sum_{{i_F}=1}^{n_F} v_{F,{i_F}}(x)
 f(x) \\
=&\ f(x),
\end{align*}
as claimed.

\end{proof}

\vskip 1pc
 Recall that an action $G$ on $X$ is \emph{free} if
 $gx\neq x$ for $g\in G$, $g\neq 1$, and $x\in X$
\vskip 1pc

\begin{cor}\label{cor-commutative}
 Let $X$ be a compact Hausdorff space and $G$ a finite group acting
on $X$. If $\alpha$ is the induced action on $C(X)$, the following
are equivalent.
 \begin{enumerate}

 \item[(i)] $G$ acts freely on $X$.

 \item[(ii)] $E : C(X)\to C(X)^\alpha$ is of
 index-finite type with $Index(E)=|G|$.

 \item[(iii)] $\alpha$ is saturated.

 \end{enumerate}
\end{cor}

\begin{proof} (i) $\Longrightarrow$ (ii) is proved in
\cite[Proposition 2.8.1]{Wa}.

  To show (ii) $\Longrightarrow$ (i),
  let $E$ be of index-finite type with $Index(E)=|G|$.
  Then from Theorem~\ref{thm-comm}, we have
  $$|G|=Index(E)=\sum_{H<G} \frac{|G|}{|H|}\chi_{X_H},$$ which
  implies that $H=\{\iota\}$ is the only subgroup of $G$ such that
  $X_H \neq \emptyset$. Hence $X=X_{\{\iota\}}$, that is, $G$ acts
  freely on $X$.   (ii) $\Longleftrightarrow$ (iii) comes from
  Theorem~\ref{thm-main-1}.
\end{proof}

\vskip 1pc
\section{Saturated actions by compact groups}
\vskip 1pc

\begin{notation}\label{notation-f}
Let $M$ be a $C^*$-algebra and $\alpha$ be an action
of a compact group $G$ on $M$. For $x,y\in M$,
 define continuous functions
$f_{x,y},\, f_{x,1}, \, f_{1,y}\in C(G,M)$ from $G$ to $M$ as
follows:
\begin{align*} f_{x,y}(t)&=x\alpha_t(y), \\
 f_{x,1}(t)  =x,  \ \
 & f_{1,y}(t)  = \alpha_t(y) \ \text{ for } t\in G.\end{align*}
Then it is easily checked that  $f_{x,y}= f_{x,1}* f_{1,y}$ and
$f_{x,y}^* =f_{y^*,x^*}.$
\end{notation}

\vskip 1pc\noindent Recall that $C(G,M)$ is a dense $*$-subalgebra
of $M\times_\alpha G$ with the multiplication and involution
defined by
\begin{align*} f*g(t)&=\int_{G} f(s)\alpha_s(g(s^{-1}t))ds,  \\
 f^*(t)&=\alpha_t(f(t^{-1})^*),\end{align*}  where $dg$ is
the normalized Haar measure (\cite[7.7]{Pe}, \cite[8.3.1]{Fi}).
Hence if $G$ is a finite group, $f_{x,y}$ can be written as
$f_{x,y}=\frac{1}{|G|}\sum_g x\alpha_g(y)\lambda_g$.

 If $\tilde M$  denotes the smallest unitization of
 $M$ (so $\tilde M=M$ if $M$ is unital),
 the function $$e:G\to \tilde M,\ \ e(s)= 1,\ \text{ for every }s\in G$$
 is a projection of the multiplier algebra
of $M\times_\alpha G$ (\cite{Ro}).

\vskip 1pc
\begin{prop}\label{prop-2} {\rm (\cite{Ro})}
Let $\alpha$ be an action of a compact group $G$ on a
$C^*$-algebra $M$. Then identifying $x\in M^\alpha$ and the
constant function in $C(G,M)$ with  the value $x$ everywhere we
see that
$$x\mapsto f_{x,1} :M^\alpha \to \ e(M\times_\alpha G)e$$ is an isomorphism
of $M^\alpha$ onto the hereditary subalgebra $\ e(M\times_\alpha
G)e$ of the crossed product $M\times_\alpha G$.
\end{prop}

\vskip 1pc

\noindent The notion of saturated action is introduced by Rieffel
for a compact group action on a $C^*$-algebra, and we adopt the
following equivalent condition as the definition.

\vskip 1pc
\begin{dfn}(Rieffel, see \cite[7.1.9 Lemma]{Ph1}) \label{def-sat}
Let $M$ be a   $C^*$-algebra  and   $\alpha$ be an action of
compact group $G$ on $M$.  $\alpha$ is said to be
 \emph{saturated}  if the linear span of $\{f_{a,b}\mid a,b\in
M\}$ is dense in $M\times_\alpha G$ (see
Notation~\ref{notation-f}). We denote $$\mathcal
J_\alpha=\overline{span}\{f_{a,b}\mid a,b\in M\}.$$
\end{dfn}

\vskip 1pc

\begin{prop}\label{prop-J}
 Let $\alpha$ be an action of a compact group $G$ on a $C^*$-algebra
$M$. Then $\mathcal J_\alpha$  is the ideal of $M\times_\alpha G$
generated by the hereditary subalgebra $e(M\times_\alpha G)e$.
Moreover $\mathcal J_\alpha=\overline{span}\{f_{a,a^*}\in
C(G,M)\mid a\in M\}$.
\end{prop}

\begin{proof}
 We first show that  $\mathcal
J_\alpha$  is an ideal   of $M\times_\alpha G$. Let $x\in C(G, M)$
and  $ a,b \in M$.  Then $x*f_{a,b}\in \mathcal J_\alpha$. Indeed,
 \begin{align*}(x*f_{a,b})(t)&=\int
x(s)\alpha_s(f_{a,b}(s^{-1}t))ds\\  &=\int x(s)\alpha_s(a)\alpha_t(b)ds\\
&=\big(\int x(s) \alpha_s(a)ds \big)  \alpha_t(b),
\end{align*} hence $x*f_{a,b}=f_{c,b}\in \mathcal
J_\alpha,$ where $c=\int x(s) \alpha_s(a)ds\in M$.  Also
$f_{a,b}^* = f_{b^*,a^*}$ implies that
 $\mathcal J_\alpha={\mathcal J_\alpha}^*$ is  an ideal of $M\times_\alpha G$.

Let $\mathcal  J:=\overline{(M\times_\alpha G)e(M\times_\alpha
G)}$ be the closed ideal generated by $e(M\times_\alpha G)e$. Now
we show that $\mathcal J_\alpha\subset \mathcal J$. From
$$(f_{a,b}*e)(t)=\int f_{a,b}(s)\alpha_s(e(s^{-1}t))\,ds=\int
a\alpha_s(b)\,ds=a \int\alpha_s(b)ds,$$ we have
$f_{a,b}*e=f_{aE(b),1}$, where $E(b)=\int\alpha_s(b)ds\in
M^\alpha.$
 Hence for $a,b,c,$ and $d$ in $M$, we have
 \begin{align*} f_{a,b}*e *f_{c,d} &=
 (f_{a,b}*e)*(f_{d^*,c^*}*e)^*\\
  & = f_{aE(b), 1}*(f_{d^*E(c^*),1})^*\\
  & =f_{aE(b),1}*f_{1,E(c^*)d}\\
  & = f_{aE(b), E(c^*)d},
 \end{align*}
which means that $f_{ax,yd}\in \mathcal J$ for any $a,d\in M$ and
$x,y\in M^\alpha$.
   Since $M^\alpha$ contains an approximate identity for $M$,
 it  follows that   $f_{a,b} \in \mathcal J$ for $a,b\in
 A$.

  For the converse inclusion $\mathcal J\subset \mathcal J_\alpha$,
 note that if $x\in C(G,M)$, then $(x*e)(t)=\int x(s)ds$ for $t\in G$.
 With notations  $x'=\int x(s) ds$ and  $x'':=\int \alpha_s(x(s^{-1}))ds\,(\in M)$,
 we see that
  \begin{align*} (x*e*y)(t)&=\int (x*e)(s)\alpha_s(y(s^{-1}t))ds\\
   &=x' \int \alpha_s(y(s^{-1}t))ds\\
   &=x' \alpha_t(\int \alpha_{s}(y(s^{-1}))ds)\\
   &=x' \alpha_t(y'')\\
   &=f_{x',y''}(t)
  \end{align*}
belongs to $\mathcal J_\alpha$  for  $x,y\in C(G,M)$.

   Finally   the following polarization identity proves the last assertion.
    $$ a\alpha_t(b)=\frac{1}{4}\sum_{k=0}^3
 i^k(b+i^ka^*)^*\alpha_t(b+i^ka^*).$$
 \end{proof}

\vskip 1pc

\section{The gauge action $\gamma$ on a graph $C^*$-algebra}

  By a (directed) graph $E$ we mean a quadruple
 $E=(E^0,E^1,r,s)$ consisting of the vertex set $E^0$, the edge set $E^1$, and
 the range, source maps $r,s: E^1\to E^0$. If each vertex of $E$ emits
only finitely many edges $E$ is called \emph{row finite}  and a
row finite graph $E$ is \emph{locally finite} if each vertex
receives only finitely many edges. By $E^n$ we denote the set of
all finite paths $\alpha=e_1\cdots e_n$ ($r(e_{i})=s(e_{i+1}),
1\leq i\leq n-1$) of \emph{length} $n$ ($|\alpha|=n$). Each vertex
is regarded as a finite path of length $0$. Then  $E^*=\cup_{n\geq
0}E^n$ is the set of all finite paths and the maps $r$ and $s$
naturally extend to $E^*$.
  A vertex $v$ is called a {\it sink} if $s^{-1}(v)=\emptyset$ and a
\emph{source} if $r^{-1}(v)=\emptyset$.

If $E$ is a row finite graph, we call a family $\{s_e, p_v\mid
e\in E^1, v\in
 E^0\}$ of operators a {\it Cuntz-Krieger(CK) $E$-family} if $\{s_e\}_e$
 are partial isometries and
 $\{p_v\}_v$ are mutually orthogonal projections such that
$$s_e^* s_e =p_{r(e)}\ \text{ and }\
 p_v=\sum_{s(e)=v} s_e s_e^*\ \text{ if }  s^{-1}(v)\neq \emptyset.$$
It is now well known that there exists a $C^*$-algebra $C^*(E)$
generated by a {\it universal} CK $E$-family $\{s_e,p_v\mid e\in
E^1, v\in E^0\}$, in this case we simply write
$C^*(E)=C^*(s_e,p_v)$. For the definition  and basic properties of
graph $C^*$-algebras, see,  for example, \cite{BHRS,BPRS,
KPR,KPRR,Ra} among others. If $\alpha=\alpha_1\alpha_2\cdots
\alpha_{|\alpha|}\, (\alpha_i\in E^1)$ is a finite path, by
$s_\alpha$ we denote the partial isometry
$s_{\alpha_1}s_{\alpha_2}\cdots s_{\alpha_{|\alpha|}}$
$(s_v=s_v^*=p_v,\text{ for } v\in E^0)$.

We will consider only  locally finite graphs and it is  helpful to
note the following properties of graph $C^*$-algebras.

\vskip 1pc
\begin{remark}\label{remark-0} \begin{enumerate}

\item[(i)]  Let $C^*(E)=C^*(s_e,p_v)$ be the graph $C^*$-algebra
associated with a row finite graph $E$, and let $\alpha,\,\beta\in
E^*$ be finite paths in $E$. Then
 $$s_\alpha^* s_\beta= \left\{%
\begin{array}{ll}
     s_\mu^*, & \hbox{if $\ \alpha=\beta\mu$ } \\
     s_\nu, & \hbox{if $\ \beta=\alpha\nu$ } \\
     0, & \hbox{otherwise.} \\
\end{array}%
\right. $$ Therefore $ C^*(E)=\overline{span}\{s_\alpha
s_\beta^*\mid \alpha,\beta\in E^*\}.$

\item[(ii)] Note that $s_\alpha s_\beta^*=0$ for $\alpha,\beta\in
E^*$  with $r(\alpha)\neq r(\beta)$.

 \item[(iii)] If $\alpha,\beta,\mu,$ and
$\nu$ in $E^n$ are the paths of  same length, $$(s_\alpha
s_\beta^*)(s_\mu
 s_\nu^*)=\delta_{\beta,\mu} s_\alpha s_\nu^*.$$
  Thus for each $n\in \mathbb N$ and a vertex $v$ in a locally
  finite graph $E$, we see that $$\overline{  span}\{s_\alpha
  s_\beta^*\mid \alpha,\beta\in E^n\text{ and }
  r(\alpha)=r(\beta)=v\}$$ is a $*$-algebra which is isomorphic to the full matrix
  algebra $M_m=(M_m(\mathbb C))$, where $m=\big|\{\alpha\in
  E^n\mid r(\alpha)=v\}\big|.$
 \end{enumerate}
\end{remark}

\vskip 1pc

Recall that the \emph{gauge action}  $\gamma$ of $\mathbb T$ on
$C^*(E)=C^*(s_e, p_v)$ is given by
$$\gamma_z(s_e)=zs_e,\ \gamma_z(p_v)=p_v,\ \ z\in
\mathbb T.$$
  $\gamma$ is well defined by the universal property of
 the CK $E$-family $\{s_e, p_v\}$.
 Since
 $$\int_{\mathbb T} \gamma_z(s_\alpha
s_\beta^*)dz=\int_{\mathbb T} z^{|\alpha|-|\beta|}(s_\alpha
s_\beta^*)dz=0,\   |\alpha|\neq |\beta|,$$ one sees that
$$C^*(E)^\gamma = \overline{span}\{s_\alpha s_\beta^*\mid
\alpha,\beta\in E^* ,\ |\alpha|=|\beta|\}.$$

\vskip 1pc  If  $\bf Z$ denotes the following graph:
 \vskip 1pc\hskip 1cm \xy
/r0.4pc/:(0,0)*+{\bullet}="V0";
(8,0)*+{\bullet}="V1";(16,0)*+{\bullet}="V2";(24,0)*+{\bullet}="V3";
(32,0)*+{\bullet}="V4"; "V0";"V1"**\crv{}
?>*\dir{>}\POS?(.5)*+!D{}; "V1";"V2"**\crv{}
?>*\dir{>}\POS?(.5)*+!D{}; "V3";"V4"**\crv{}
?>*\dir{>}\POS?(.5)*+!D{};"V2";"V3"**\crv{}
?>*\dir{>}\POS?(.5)*+!D{}, (36,0)*+{\cdots  ,},(-4,0)*+{\cdots},
(-10,0)*+{\bf Z :},
(0,-2)*+{-2},(8,-2)*+{-1},(16,-2)*+{0},(24,-2)*+{1},(32,-2)*+{2},
\endxy

\vskip .5pc \noindent then   $C^*(\bf Z)$ is isomorphic to the
$C^*$-algebra $\mathcal K$ of compact operators on an infinite
dimensional separable Hilbert space,
  hence $C^*(\bf Z)$ is itself a simple AF algebra.
  But $C^*(\bf Z)^\gamma$ coincides with the commutative subalgebra
$\overline{span}\{s_\alpha s_\alpha^*\mid \alpha\in \bf Z^*\}$
which is far from being simple, and thus  we know that the
simplicity of $C^*(E)$ does not imply that of $C^*(E)^\gamma$ in
general.

\vskip 1pc In \cite{KP},  the  {\it Cartesian product} of two
graphs $E$ and $F$  is defined to be the graph $E\times
F=(E^0\times F^0,E^1\times F^1,r,s)$,
 where $r(e,f)=(r(e),r(f))$ and $s(e,f)=(s(e),s(f))$.
  Since the graph ${\bf Z}\times E$
  has no loops for any row-finite graph $E$,
  we know that $C^*(E)^\gamma$ is an AF algebra (\cite{KPR})
  by the following proposition.

\vskip 1pc
\begin{prop}\label{prop-PK} {\rm (\cite{KP})} Let $E$ be a row
finite graph with no sources. Then the following hold:
\begin{enumerate}
\item[(a)] $C^*(E)^\gamma$ is stably isomorphic
 to $C^*(E)\times_\gamma \mathbb T.$

\item[(b)] $C^*(E)\times_\gamma \mathbb T\cong C^*({\bf Z}\times
E)$.
\end{enumerate}

\end{prop}

\vskip 1pc
  Now we show that a gauge action is saturated.
  For this, note that the linear span of the continuous functions
  of the form
$$t\mapsto f(t) x,\ \ f\in C(G),\, x\in A$$  is dense in $C(G,
A)$ \cite[7.6.1]{Pe}. Hence by Remark~\ref{remark-0}(i), one sees
that
 \begin{eqnarray}\label{graph-alg} C^*(E)\times_\gamma \mathbb T=
  \overline{span}\{z^n s_\alpha s_\beta^*\mid \alpha,\ \beta\in E^*\, \ n\in \mathbb
  Z\}.
 \end{eqnarray}

\vskip 1pc

\begin{thm}\label{thm-main-2} Let $E$ be a locally finite graph with no
sinks and no sources. Then  the gauge action $\gamma$ on $C^*(E)$
is saturated.
\end{thm}

\begin{proof}  We show that   $ \mathcal
J_\gamma=C^*(E)\times_\gamma \mathbb T$. By (\ref{graph-alg})
    it suffices to see that
  $$z^n s_\alpha s_\beta^* \in
 \mathcal J_\gamma \ \text{ for all } \alpha,\beta\in E^*,\ n\geq 0$$
 (because $
z^{-n} s_\alpha s_\beta^*=(z^n s_\beta s_\alpha^*)^*$ for $n\geq 0$).

 Now fix $\alpha,\ \beta\in E^*$ and $n\geq 0$.
 Put $l=n-(|\alpha|-|\beta|).$ There are two cases.
\begin{enumerate}

\item[(i)] $l\geq 0$:
 One can choose a path $\mu$ such that  $|\mu|=l$ and
$r(\mu)=s(\alpha)$. Then
$$ z^n s_\alpha
s_\beta^* = z^{l+|\alpha|-|\beta| } s_\mu^*  s_\mu s_\alpha
s_\beta^*= s_\mu^*\gamma_z(s_{\mu \alpha} s_\beta^*)=
f_{s_\mu^*,s_{\mu \alpha} s_\beta^*}(z),$$ where the function
$f_{s_\mu^*,s_{\mu \alpha} s_\beta^*}$ belongs to $\mathcal
J_\gamma$.

 \item[(ii)] $l<0$: Choose a path $\nu$ with
  $|\nu|=|\beta|+n$ and $r(\nu)=r(\alpha)$. With $a=s_\alpha
  s_\nu^*$,  $b=s_\nu s_\beta^*$, we have $f_{a,b}\in \mathcal
J_\gamma$ and
  $$z^n s_\alpha s_\beta^*= s_\alpha s_\nu^*\gamma_z(s_\nu s_\beta^*)=
   f_{a,b}(z).$$
  \end{enumerate}
\end{proof}

\noindent {\em Acknowledgements.} The first author would like to
thank Hiroyuki Osaka and Tamotsu Teruya for valuable discussions.

\vskip 2pc

Keywords: Finite dimensional Hopf $*$-algebra; saturated action;
conditional expectation of index-finite type.

\end{document}